\documentclass{amsart}
\usepackage{amssymb}
\usepackage[utf8]{inputenc}
\usepackage[english]{babel}
\usepackage{latexsym,cite,mathrsfs}
\usepackage{amsfonts}
\usepackage{amssymb, amsthm}
\usepackage{xcolor}

\newtheorem{theorem}{Theorem}[section]
\newtheorem{lemma}[theorem]{Lemma}
\newtheorem{proposition}[theorem]{Proposition}
\newtheorem{corollary}[theorem]{Corollary}

\newtheorem{definition}[theorem]{Definition}

\theoremstyle{remark}
\newtheorem{remark}[theorem]{Remark}

\def\({\left(}
\def\){\right)}
\def\<{\left\langle}
\def\>{\right\rangle}

\def\si{{\sigma}}

\def\R{{\mathbb R}}
\def\N{{\mathbb N}}
\def\G{{\mathcal G}}
\def\F{{\mathcal F}}

\def\Tend#1#2{\mathop{\longrightarrow}\limits_{#1\rightarrow#2}}

\numberwithin{equation}{section}

\usepackage{color,enumitem,graphicx}
\usepackage[colorlinks=true,urlcolor=blue,
citecolor=red,linkcolor=blue,linktocpage,pdfpagelabels,
bookmarksnumbered,bookmarksopen]{hyperref}

\usepackage{verbatim} 

\usepackage[hyperpageref]{backref}

\begin{document}

\title[2D NLS with combined nonlinearities and harmonic potential]{On
  ground states for the 2D Schr\"odinger equation with 
  combined nonlinearities and harmonic potential}

\author[R. Carles]{R\'emi Carles}
\address{Univ Rennes, CNRS\\ IRMAR - UMR 6625\\ F-35000
  Rennes, France}
\email{Remi.Carles@math.cnrs.fr}

\author[Y. Il'yasov]{Yavdat Il'yasov}

\address{Institute of Mathematics\\ Ufa Federal Research Centre, RAS\\
	Chernyshevsky str. 112, 450008 Ufa\\ Russia}
\email{ilyasov02@gmail.com}

\begin{abstract}
  We consider the nonlinear Schr\"odinger equation with a harmonic
  potential in the presence of two combined energy-subcritical power
  nonlinearities. We assume that the larger power is
  defocusing, and the smaller power is focusing.
  Such a framework includes physical models, and
  ensures that finite energy solutions are global in time. We address
  the questions of the existence and the orbital stability of
  the set of  standing
  waves. Given the mathematical features of the equation (external
  potential and inhomogeneous nonlinearity), the set of parameters for
  which standing waves exist in unclear. In the two-dimensional case,
  we adapt the method of fundamental frequency solutions, introduced by the second
  author in the higher dimensional case without potential. This makes
  it possible to describe accurately the set of fundamental frequency
  standing waves and ground states, and to prove its orbital stability.
\end{abstract}
\thanks{The first author was supported by Centre Henri Lebesgue,
  program ANR-11-LABX-0020-0. The second author was supported by the
  Russian Science Foundation (No.22-21-00580).} 
\maketitle

\section{Introduction}\label{sec:intro}
We consider the  nonlinear Schr\"{o}dinger (NLS) equation   with
harmonic potential and  combined power-type nonlinearities  
\begin{equation}\label{Sch}
	i\partial_t\psi+ \Delta \psi =|x|^2\psi -\mu|\psi|^{p-2}\psi+|\psi|^{q-2}\psi, ~~(t,x) \in \mathbb{R}^+\times\mathbb{R}^N,  
\end{equation}
where $\psi$ is a complex-valued function of $(t, x)$, $p, q \in
(2,+\infty)$ ($p\not = q$), $N\ge 1$ and $\mu>0$ (homogeneity
considerations show that the positive coefficient in front of the last
nonlinearity can be chosen equal to one).   
For
\[2<p,q<\frac{2N}{(N-2)_+}=:2^*,\]
the Cauchy problem for \eqref{Sch} with the initial value  $\psi_0 \in
\Sigma$ is locally well-posed and has a unique local solution $\psi
\in C([0, T(\psi_0)),\Sigma) \cap C^1([0, T(\psi_0),\Sigma^*),$ for
some $T(\psi_0)>0$, where $\Sigma$ denotes the domain of the harmonic
oscillator,

\[\Sigma=\{f\in H^1(\R^N),\ x\mapsto
  xf(x)\in L^2(\R^N)\},\]

characterized by the norm
\[
\|u\|_{\Sigma}^2=\int_{\R^N} (1+|x|^2)|u|^2+|\nabla u|^2=\|u\|_{L^2}^2 +
\<\(-\Delta+|x|^2\)u,u\>.
\]
See e.g. \cite{CazCourant}. 
In view of the uncertainty principle (see e.g. \cite{Weinstein83}),
\begin{equation}\label{eq:uncertainty}
  \|u\|_{L^2}^2\le \frac{2}{N}\|xu\|_{L^2}\|\nabla u\|_{L^2},
\end{equation}
we may remove the $L^2$ norm of $u$ from the definition of
$\|u\|_{\Sigma}$.  In physics, the potential in \eqref{Sch} corresponds 
to some confinement (see e.g. \cite{DGPS}), and the nonlinearity is typically cubic-quintic:
the quintic part is defocusing ($q=6$), and the cubic part is focusing
($p=4$ and $\mu>0$); see
e.g. \cite{AbSa05,Malomed19}. The combination of the harmonic
potential and the cubic-quintic nonlinearity can be found e.g. in
\cite{DWWZL11,WZM12}. From a mathematical point of view, the
combination of a nonlinearity which is not homogeneous, and of the
presence of an external potential, requires a specific approach. In
the case where  the external potential is smooth and decaying at
infinity,  the large time behavior of ground states was described
precisely under smallness assumptions, and possibly for specific
nonlinearities and/or potentials; see e.g. \cite{CuccagnaMaeda15,DengSofferYao18,GangZhou07,NaumkinRaphael20,SofferWeinstein04} and references
therein.  The heart of the present paper addresses
the two-dimensional case $N=2$. 
\smallbreak

Solutions to \eqref{Sch} obey the \textit{energy conservation law}:
\[
E \equiv H_\mu(\psi(t)):=\int_{\R^N} \left(\frac{1}{2} |\nabla \psi|^2 +\frac{|x|^2}{2}|\psi|^2 - \frac{\mu}{p}| \psi|^p +\frac{1}{q}| \psi|^q\right)dx,
\]
and the \textit{mass (charge, particle
numbers) conservation law}:
\[
\alpha\equiv Q(\psi(t)):=\frac{1}{2}\int_{\R^N} | \psi|^2 dx. 
\]
The assumption $2<p<q<2^*$ ensures global
  well-posedness in $\Sigma$ for \eqref{Sch}, since the defocusing
  nonlinearity dominates the focusing one (see also
  Lemma~\ref{lemBOUND} below). We shall mainly consider
  this assumption regarding $p$ and $q$, an assumption which includes
  the cubic-quintic nonlinearity in 2D. We mention in passing that on
  the other hand, the Cauchy problem for the 3D cubic-quintic equation
  with harmonic potential is not straightforward, and the cubic
  nonlinearity is energy-critical; see \cite{KiViZh09} for the Cauchy problem
  when $\mu=0$, and \cite{Zhang06} for the 3D cubic-quintic equation without
  harmonic potential. On the other hand, if $2+\frac{4}{N}<p\le 2^*$ and
  $2<q<p$, then finite time blow-up is possible, as shown in
  \cite{TaoVisanZhang} (see also e.g. \cite{CMZ16,Feng2018}),  and the
  issue of stability or instability of solitary waves in that case
  was studied in \cite{FuHa21,Hayashi-p}. 
  \smallbreak
  
We study the existence and stability of standing waves $\psi_\lambda(t,x)=e^{i\lambda t}u(x)$ of \eqref{Sch}, where  the amplitude function $u$    satisfies
\begin{equation}\label{1S}
	-\Delta u +|x|^2u+\lambda u-\mu|u|^{p-2}u+|u|^{q-2}u =0, ~~x \in \mathbb{R}^N.   
\end{equation}
Here $\lambda \in \mathbb{R}$ is the \textit{frequency} of the standing wave. 
\smallbreak

There are at least two motivations to study \eqref{1S},
apart from the fact that \eqref{Sch} appears in some physical models
(e.g. \cite{AbSa05,Malomed19}).  First, due to the fact that the
nonlinearity in \eqref{1S} is not homogeneous, the range of possible
values of the parameters $\lambda$ and $\mu$, in order for \eqref{1S}
to have nontrivial solutions, is not clear. This aspect is explained in
more details below. In particular, we could not find
  any reference addressing the existence of ground state solutions to
  \eqref{1S}.  Second, in the case without
potential, the set of stationary solutions was studied recently in
\cite{ilyasFFS} by using a generalization of the Rayleigh quotient to
the nonlinear framework \cite{Ily17}. However, the approach presented
in \cite{ilyasFFS} excludes the two-dimensional case (as well as the
one-dimensional case, a situation where many specific techniques are
available, see e.g. \cite{BL83a,IlievKirchev93}). As we will see, the
introduction of the harmonic potential removes this restriction. In
principle, Equation~\eqref{1S} can be studied via the nonlinear
generalized Rayleigh quotient method in any space dimension $N$, but
computations are more explicit in the 2D case, this is why we choose
to focus on that case here. 
\smallbreak
The orbital stability of the ground states of many equations can be
often investigated using a Lyapunov functional, determined by the action functional $S_{\lambda,\mu}(u)$ restricted to the manifold of functions $u$ with fixed mass integral $m= Q(u)$.   Such an approach was first applied to Korteweg-de Vries (KdV) solitons by Benjamin \cite{benj}, and later, to three-dimensional ion-acoustic
solitons in strongly magnetized plasma by Zakharov and Kuznetsov
\cite{ZaharKuzn}.   
The general theorem on orbital stability of the so-called (set of)
mass-prescribed solutions of nonlinear problems based on this idea was
proved by T.~Cazenave and P.-L.~Lions \cite{CaLi82}. However,  the
orbital stability of mass-prescribed solutions does not always entail
the orbital stability of ground states of nonlinear Schr\"odinger
equations, typically when the nonlinearity is not homogeneous, or  in
the presence of an external potential. In other words, the notions of
orbital stability of (the set of) mass-prescribed solutions and
orbital stability of (action minimizing) ground states coincide in the
case of homogeneous nonlinearities without potentials, as explained in
\cite{CaLi82} (see also \cite{CazCourant}), but apart from this case
(and the very singular case of a logarithmic nonlinearity, where the
role of the mass is specific, see \cite{Caz83,Ar16}),
the approaches introduced to prove orbital stability differ; see
\cite{Weinstein85,Weinstein86CPAM,GSS87,IlievKirchev93}, and also \cite{BGR15},
in the case of orbital stability of ground states.
\smallbreak

We emphasize that even in the case without potential, the notion of
ground state solutions to \eqref{1S} requires some care. The notion
which seems to be privileged in physics consists in minimizing the
energy under the constraint of a fixed mass $M>0$,
\begin{equation}\label{eq:constrained}
  \inf\{ H_\mu(u);\quad Q(u)= M\}.
\end{equation}
The other approach consists in seeking a minimizer of the action 
\begin{equation*}
S_{\lambda,\mu}(u):=H_\mu(u)+\lambda Q(u) ,
\end{equation*}
for a prescribed frequency $\lambda\in \R$; see \cite{BL83a,BGK83}. In
this article, we consider the following notion (see also
\cite{JeanjeanLu-p} and references therein):
\begin{definition}
  For a given $\lambda \in \mathbb{R}$, a solution $\hat{u}_{\lambda}$
  of \eqref{1S} is said to be a \emph{ground state} if $S_{\lambda,\mu}(\hat{u}_{\lambda})\leq S_{\lambda,\mu}(w)$, for any $w \in\Sigma\setminus \{0\}$ such that $DS_{\lambda,\mu}(w)=0$. For $\mu>0$, we denote by 
\[
\G_{\lambda,\mu}:=\left\{u \in \Sigma\setminus\{0\}: S_{\lambda, \mu}(u)= \hat{S}_{\lambda,\mu},~D S_{\lambda, \mu}(u)=0\right\}
\]
the set of the ground states of \eqref{1S}, where 
\[\hat{S}_{\lambda,\mu}:=\min\left\{S_{\lambda,\mu}(u):~D S_{\lambda,
    \mu}(u)=0,~ u \in \Sigma\setminus\{0\}\right\}\]
 is called the \emph{ground  level}. 
\end{definition}
The relation between the two notions is not clear in general: if $u$ is a minimizer for
\eqref{eq:constrained}, then there exists a Lagrange multiplier so
that $u$ solves \eqref{1S} for some $\lambda\in \R$. In the case of an
homogeneous nonlinearity (typically $-|\psi|^{p-2}\psi$), a scaling
argument shows that any $\lambda>0$ can be obtained, and the solution
is also a minimizer of the action, see \cite{CaLi82} (see also
\cite{CazCourant}). In the case of combined power nonlinearities, that
is \eqref{Sch} without potential, it is known that not every
minimizer of the action is a solution to \eqref{eq:constrained}, see
\cite{CKS-p,LewinRotaNodari20}. It has been proven recently that for a
large class of 
nonlinearities (much larger than the sum of two power nonlinearities),
every solution to \eqref{eq:constrained} minimizes the action; see
\cite{JeanjeanLu-p}. 
\smallbreak

We now discuss more precisely why  it is not
obvious to describe the set of accessible values of $\lambda$ and
$\mu$ to solve \eqref{1S}. Solutions to 
\begin{equation}
  \label{eq:Vgen}
  -\Delta u + V(x) u + \lambda u = f(|u|^2)u,
\end{equation}
were constructed in \cite{PL284b} for potentials $V$ having a finite limit as
$|x|\to \infty$. In\cite{FloerWeinstein1986}, the case $N=1$ with $V$
bounded and 
a cubic nonlinearity ($f(y) = y$) was considered, thanks to
Lyapunov-Schmidt method, allowing to invoke perturbation arguments
near the ground state associated to the case $V=0$. These results were
extended in 
\cite{RoseWeinstein}, again in the case of a bounded potential,
going to zero at infinity, for an homogeneous nonlinearity $f(y) =
y^{\sigma}$. By bifurcation arguments, the authors show that for
$\lambda$ near $\lambda_*>0$ ground state eigenvalue of $\Delta
-V$, \eqref{eq:Vgen} has a solution. The case of an harmonic
potential $V(x)=|x|^2$ was considered in  
\cite{KavianWeissler} for the case of an $L^2$-critical nonlinearity
$f(y)=y^{2/N}$, thanks to a mountain pass lemma.  Additional results
were obtained regarding the 
existence and stability of standing waves for unbounded potentials $V$
in e.g. \cite{Reika2001,FuOh03}. The following result is very similar
to \cite[Theorems~1.3 and 1.4]{KavianWeissler}, and is proven in
Appendix~\ref{sec:constrained}. 
\begin{proposition}\label{prop:constrained}
   Let $N\ge 1$ and $2<p<2^*=\frac{2N}{(N-2)_+}$.\\
$1.$  If $\lambda>-N$, then there exists a solution $u\in
\Sigma\setminus\{0\}$ to
\begin{equation}\label{eq:homodefoc}
  -\Delta u+|x|^2u +\lambda u = |u|^{p-2}u.
\end{equation}
$2.$ If $\lambda<-N$, then there exists a solution $u\in
\Sigma\setminus\{0\}$ to
\begin{equation}\label{eq:homofoc}
  -\Delta u+|x|^2u +\lambda u = -|u|^{p-2}u.
\end{equation}
$3.$ If $\lambda>-N$, then there exists
   $\mu>0$ and a solution $u\in
\Sigma\setminus\{0\}$ to \eqref{1S}. 
\end{proposition}

Apart from the existence of standing wave, and their classification,
the standard question is their dynamical stability: due to the
invariances of the equation, the relevant notion is the notion of
orbital stability, see e.g. \cite{CazCourant,BGR15}. In the present
case, the  harmonic 
potential removes the ``usual'' space-translation invariance, and,
along with time-translation invariance, only
the gauge invariance remains: if $\psi$ solves \eqref{Sch}, then so
does $\psi e^{i\theta}$, for any constant $\theta\in \R$. Like for the
existence of standing waves, at least two approaches are commonly used
in order to study the question of orbital stability: the stability of
a set of standing waves (typically, the minimizers associated to
\eqref{eq:constrained}, as in \cite{CaLi82}), and the stability of a
given standing wave (as in
\cite{Weinstein85,Weinstein86CPAM,GSS87,IlievKirchev93}; 
see also \cite{BGR15}).  
Like in \cite{ilyasFFS}, our main result is in the
spirit of the former notion: we prove the orbital stability of the set
of fundamental frequency solutions, defined below. Note that by
resuming the approach of e.g. \cite{CaSp21} in the 2D case (the 3D
case is different, as the quintic nonlinearity becomes
energy-critical), based on the method introduced in \cite{CaLi82}, it
is easy to prove that the set of energy minimizers 
under a mass constraint (the minimizers of \eqref{eq:constrained}) is
orbitally stable. But then again, we do not know for which values of
$\mu$ this set is nonempty, and which are the corresponding Lagrange
multipliers $\lambda$.  One of the main features of this
article is to define quantities that make it possible to characterize the
set of $\mu$'s and $\lambda$'s for which ground states exist, and
their set is stable.

\smallbreak

Let us state our main result. Introduce,
\begin{equation}\label{eq:muo}
	\hat{\mu}^0=C_{p,q}\inf_{u \in \Sigma\setminus \{0\}}
\frac{\(\int|\nabla u|^{2}\cdot \int |x|^2 |u|^{2}\)^{\frac{q-p}{2(q-2)}}(\int
  |u|^{q})^\frac{p-2}{(q-2)}}{ \int
  |u|^{p}}.
\end{equation}
for some $C_{p,q}>0$ introduced in Section~\ref{sec:rayleigh}: $\hat{\mu}^0>0$ (see Proposition~\ref{prop:posit.bound} below). 
\begin{theorem}\label{theo:main}
 Let $N=2$, $2<p<q<\infty$, and $\mu>\hat{\mu}^0$. Then there exists $ \hat{\lambda}_\mu^{*}>0$ such that: 
\par $(1^o)$ For any $\lambda\in [0,
\hat{\lambda}_\mu^{*}]$,  Equation~\eqref{1S} has a ground state $\hat{u}_\lambda \in \Sigma$.
\par $(2^o)$ $S_{\lambda,\mu}(\hat{u}_\lambda)<0$ for $\lambda\in [0,
\hat{\lambda}_\mu^{*})$ and $S_{\lambda,\mu}(\hat{u}_\lambda)|_{\lambda=\hat{\lambda}_\mu^{*}}=0$. 
\par $(3^o)$ Up to a gauge transform (replace $\hat{u}_\lambda$ by
$e^{i\theta}\hat{u}_\lambda$ for some constant $\theta\in \R$), $\hat{u}_\lambda$ is positive, radially symmetric, and non-increasing in $r=|x|$.
\par $(4^o)$ For any $\lambda\in [0,
\hat{\lambda}_\mu^{*}]$, the set of ground states $\G_{\lambda,\mu}$ of \eqref{1S} is
 orbitally stable: for any $\varepsilon > 0$, there exists $\delta> 0$ such that if 
$\psi_0\in \Sigma$, satisfies
\[\inf_{\phi \in \G_{\lambda,\mu}}\|\psi_0
  -\phi\|_{\Sigma}< \delta,\]
then
\[	\sup_{t\in \R}\inf_{\phi \in \G_{\lambda,\mu}}\|\psi(t)-\phi\|_{\Sigma}<\varepsilon,
\]
where $\psi$ is the solution to \eqref{Sch} such that $\psi_{\mid
  t=0}=\psi_0$. 
\end{theorem}
The proof of the theorem is based on development of the method of fundamental frequency solutions, introduced by the second   author in \cite{ilyasFFS} in the higher dimensional case without potential. We deal with the so-called \textit{prescribed action solution} of
\eqref{1S}, i.e., a function $u^S \in \Sigma$  which for a given
action $S \in \R$  satisfies
\begin{equation}\label{EqS}
	S_{\lambda,\mu}(u^S)=S ~~\mbox{and}~~ DS_{\lambda,\mu}(u^S)=0,
\end{equation}
 with  some $\lambda\ge 0$. 
As mentioned above, an alternative formulation which has also been actively investigated
over the last decades consists in finding the solution $u$ to
\eqref{1S} having \textit{ prescribed mass}, via
\eqref{eq:constrained}, while $\lambda$ and $S$ are unknown. Note that
the frequency $\lambda$ can be also considered as a value of the
following  conserved quantity 
\begin{equation}\label{conservL}
	\lambda = \Lambda^S_\mu(\psi(t)):=\frac{S-H_\mu(\psi)}{Q(\psi)}.
\end{equation}

\begin{definition}\label{def:FFsol}
  For a given $S \in \mathbb{R}$, a solution $\hat{u}^S$
  of \eqref{1S} is said to be a fundamental frequency solution (respectively,
$e^{i\hat{\lambda} t}\hat{u}$ is said to be a \textit{fundamental
  frequency standing wave} of \eqref{Sch}) with a \textit{fundamental
  frequency} $\hat{\lambda}^S_\mu$ if  
\[
\hat{\lambda}^S_\mu=\Lambda^S_\mu(\hat{u})\ge \Lambda^S_\mu(w)
\] 
for any  $w \in \Sigma\setminus \{0\}$  such that $D\Lambda^S_\mu(w)=0$. 	For $S \in \mathbb{R}$, $\mu>0$, we denote by 
\[
\F_{\mu}^S:=\left\{u \in \Sigma\setminus\{0\}: \Lambda^S_\mu(u)= \hat{\lambda}^S_\mu,~D  \Lambda^S_\mu(u)=0\right\}
\]
the set of the fundamental frequency solutions of \eqref{1S}, where 
\[\hat{\lambda}^S_\mu:=\max
  \left\{\Lambda^S_\mu(u):~D \Lambda^S_\mu(u)=0,~ u \in \Sigma\setminus\{0\}\right\}\]
 is called the \emph{fundamental  frequency}. 
\end{definition}

We have the following result on the existence and orbital stability of
the set of   fundamental
  frequency solutions:
\begin{theorem}\label{theo:main2}
 Let $N=2$ and $2<p<q<\infty$. Then for any given action value $S \in \R$, there exists an extremal value
$\hat{\mu}^S\ge 0$ such that: 
\par $(1^o)$ For any $\mu > \hat{\mu}^S$,  Equation~\eqref{1S} has a
fundamental frequency solution  $\hat{u}^S_\mu \in \Sigma \setminus
\{0\}$. In the case $S\le 0$, the same is true under the relaxed
assumption $\mu \ge \hat{\mu}^S$.
\par $(2^o)$  Up to a gauge transform (replace $\hat{u}^S_\mu$ by
$e^{i\theta}\hat{u}^S_\mu$ for some constant $\theta\in \R$),  $\hat{u}^S_\mu$ is positive, radially symmetric, and non-increasing in $r=|x|$.
\par $(3^o)$ If $S\leq 0$, then $\hat{\mu}^S> 0$, and for any $\mu \geq  \hat{\mu}^S$, the fundamental frequency solution  $\hat{u}^S_\mu$ is a ground state of \eqref{1S} with $\lambda= \Lambda^S_\mu(\hat{u}^S_\mu)$. Moreover,  $\lambda\in [0,
\hat{\lambda}_\mu^{*})$. 
\par $(4^o)$ If $S\leq 0$, then for any  $\mu \geq  \hat{\mu}^S$, the set of fundamental frequency solutions $\F_{\mu}^S$ of \eqref{1S} is
 orbitally stable: for any $\varepsilon > 0$, there exists $\delta> 0$ such that if 
$\psi_0\in \Sigma$, satisfies
\[\inf_{\phi \in \F_{\mu}^S}\|\psi_0
  -\phi\|_{\Sigma}< \delta,\]
then
\[	\sup_{t\in \R}\inf_{\phi \in \F_{\mu}^S}\|\psi(t)-\phi\|_{\Sigma}<\varepsilon,
\]
where $\psi$ is the solution to \eqref{Sch} such that $\psi_{\mid
  t=0}=\psi_0$. 
\end{theorem}	
 We obtain the existence of  the fundamental frequency solutions by applying  the nonlinear generalized Rayleigh quotient method (NG-Rayleigh quotient method) \cite{Ily17, ilyasFFS, Ilyas21}  to
\eqref{1S} with $\Lambda^S_\mu(u)$ as the Rayleigh quotient. 
 \bigbreak
\begin{remark}
	We conjecture that Equation~\eqref{1S} has no ground state solution with positive action $S>0$. Furthermore, we anticipate the fundamental frequency solutions  $\hat{u}^S_\mu$ of Equation~\eqref{1S} for $S>0$ are not orbitally stable. 
\end{remark}
\subsection*{Notations}

Let $(\alpha_n)_{n\in\N}$ and
  $(\beta_n)_{n\in\N}$ be two families of positive real numbers. 
\begin{itemize}
\item We write $\alpha_n \ll \beta_n$ if
$\displaystyle \limsup_{n\to +\infty}\alpha_n/\beta_n =0$.
\item We write $\alpha_n \lesssim \beta_n$ if 
$\displaystyle \limsup_{n\to +\infty}\alpha_n/\beta_n <\infty$.
\item We write $\alpha_n \approx \beta_n$ if $\alpha_n \lesssim
  \beta_n$ and $\beta_n \lesssim \alpha_n$. 
\end{itemize}


\section{The nonlinear generalized Rayleigh quotients }\label{sec:rayleigh}

For $u \in \Sigma$ and $\sigma>0$, we denote $u_\sigma(x):=u(x/\sigma)$, $x \in \mathbb{R}^N$,   and 
\begin{align*}
 &T(u):=\int|\nabla u|^{2},~L(u):=\int |x|^2 |u|^{2},~~Q(u):=\frac{1}{2}\int
  |u|^{2},\\
  &A(u):= \int |u|^{p},~~B(u):=\int |u|^{q}.
\end{align*}
With these notations we have
\[
S_{\lambda,\mu}(u):=\frac{1}{2}T(u)+\frac{1}{2}L(u)+\lambda
Q(u) -\mu \frac{1}{p} A(u) +\frac{1}{q}B(u). 
\]
and
\begin{align*}
	&T(u_\sigma)=\sigma^{N-2}T(u), ~ L(u_\sigma)=\sigma^{N+2}L(u),~Q(u_\sigma)=\sigma^{N}Q(u),\\&~
	A(u_\sigma)=\sigma^{N}A(u),~B(u_\sigma)=\sigma^{N}B(u).
\end{align*}
By a weak solution of \eqref{1S} we mean a critical point $u \in
\Sigma$ of $S_{\lambda,\mu}(u)$ on $\Sigma$. Observe that if $u \in
\Sigma$ is a weak solution of \eqref{1S}, then the Br\'ezis-Kato Theorem
\cite{BK} and the $L^\gamma$ estimates for the elliptic problems
\cite{ADN} yield  $u \in W^{2,\gamma}_{\rm loc}(\R^N)$, for
any $\gamma \in (1,\infty)$, and whence by the regularity theory of
the solutions of the elliptic problems $u \in C^{2}(\R^N)\cap H^1(\R^N)$
(see e.g. \cite{ADN}).

This implies that  any weak solution $u$
of \eqref{1S} satisfies the Pohozaev identity \cite{poh} 
\begin{equation}\label{Poh}
P(u):=\frac{N-2}{2}T(u)+\frac{N+2}{2}L(u)+\lambda
N Q(u) -\mu \frac{N}{p} A(u) +\frac{N}{q}B(u)=0.
\end{equation}
See e.g. \cite[Section~2.1]{BL83a} for details (the assumption $u\in
\Sigma$ makes it possible to adapt the proof in order to include the
harmonic potential, by repeating the argument on top of p.~321 in
\cite{BL83a}). 
For $S\in \R$, introduce the so-called \textit{action-level Rayleigh quotient}
\begin{equation}\label{eq:alRq}
	\Lambda^S_{\mu}(u):= \frac{S-\frac{1}{2}T(u)-\frac{1}{2}L(u)+\mu \frac{1}{p}A(u) -\frac{1}{q}B(u)}{Q(u)}.
\end{equation}

Notice that for any $S \in \R$ and $\lambda \in \R$,
\begin{itemize}
\item $\Lambda^S_\mu(u)=\lambda~\Leftrightarrow~ S_{\lambda,\mu}(u)=S$,
\item $D\Lambda^S_\mu(u)=0$ with $\Lambda^S_\mu(u)=\lambda
  ~\Leftrightarrow~DS_{\lambda,\mu}(u)=0$, i.e., $u$ satisfies \eqref{1S}. 
\end{itemize}
Let $u \in \Sigma\setminus \{0\}$, $S\in \R$, $\sigma>0$, consider
\begin{equation}\label{sigmaFun}
		\lambda_u(\sigma):=\Lambda^S_{\mu}(u_\sigma)=\frac{\sigma^{-N}S-\sigma^{-2}\frac{1}{2}T(u)-\sigma^{2}\frac{1}{2}L(u)+\mu \frac{1}{p}A(u) -\frac{1}{q}B(u)}{Q(u)}.
	\end{equation}
Then for $u \in \Sigma$,
\[
\frac{d}{d\sigma} \lambda_u(\sigma)=0 ~\Leftrightarrow~ -NS+\sigma^{N-2}T(u)-\sigma^{N+2}L(u)=0.
\]
The last equation, 
\begin{itemize}
	\item For $N>2$, has a unique solution $\sigma_S(u)>0$ for any
          $u \in \Sigma$ and $S<0$. There are two positive solutions
          if $S>0$ is not too large.
	\item For $N=2$, has a unique solution $\sigma_S(u)>0$ for any
          $u \in \Sigma$ and $S\leq \frac{1}{2}\ T(u)$.
        \item For $N=1$, has a unique solution $\sigma_S(u)>0$ under
          suitable relations between $T(u)$, $L(u)$ and $S$. 
\end{itemize}

Note that for $N\not =2$, the formula(s) for $\sigma_S(u)$ may not be
easy to write, when possible. 
\smallbreak

\emph{From now on, we assume $N=2$.}
\smallbreak

 Then 
\[
\sigma_S(u)=\(\frac{T(u)-2S}{L(u)}\)^{1/4}, ~~\forall u\in
\Sigma\setminus\{0\},~~\text{such that}~~S\leq \frac{1}{2}\ T(u).
\]
Hence, in accordance with \cite{ilyasFFS}, we have the following
nonlinear generalized Rayleigh quotient (NG-Rayleigh quotient)
\begin{equation}\label{LCN}
   	\lambda^S_\mu(u ):=\Lambda^S_\mu(u_\sigma
        )|_{\sigma=\sigma_S(u)}=-\frac{\left((T(u)-2S)^{1/2}L(u)^{1/2}-\mu
          \frac{1}{p}A(u)+\frac{1}{q}B(u)\right)}{Q(u)}.
\end{equation}
Observe, when $T(u)> 2S$,
\begin{equation*}
	\lambda^S_\mu(u )=\max_{\sigma>0}\Lambda^S_\mu(u_\sigma
        ), 
\end{equation*}
when $T(u)=2S$,
\begin{equation*}
	\lambda^S_\mu(u )=\sup_{\sigma>0}\Lambda^S_\mu(u_\sigma
        )= \frac{\frac{\mu}{p}A(u) -\frac{1}{q}B(u)}{Q(u)}, 
\end{equation*}
and when $T(u)<2S$ (a case which may occur only for $S>0$), 
\begin{equation*}
	\sup_{\sigma>0}\Lambda^S_\mu(u_\sigma
        )=\infty.
\end{equation*}
Denote, for  $S\in \mathbb{R}$,
\[
\Sigma_S:=\{u \in \Sigma;\ T(u)>2S\}.
\]
\begin{corollary}\label{cor:diffSigma}
If $u \in \Sigma$ satisfies \eqref{1S} and $S:=S_{\lambda,\mu}(u)$, then $\sigma_S(u)=1$ and $\frac{d}{d\sigma} \lambda_u(\sigma)|_{\sigma=1}=0$. 
\end{corollary}
\begin{proof} The equalities $\frac{d}{d\sigma} \lambda_u(\sigma)|_{\sigma=1}=0$ and $\sigma_S(u)=1$ are consequences of Pohozaev identity \eqref{Poh}. Indeed, $S_{\lambda,\mu}(u)-
\frac{1}{2}P(u)=\frac{1}{2}T(u)-\frac{1}{2}L(u)$ and therefore, $T(u)-2S=L(u)$,
i.e., $\sigma_S(u)=1$. The second equality follows similarly. 	
\end{proof}

\begin{lemma} \label{CritR}
Let $S\in\mathbb R$. Then  $D\lambda^S_\mu(u )=0$, $ \lambda^S_\mu(u )=\lambda$,
$\sigma_S(u)=1$  if and only if  $u \in \Sigma$ is a weak solution
of \eqref{1S} with prescribed action $S$. 
\end{lemma}
\begin{proof}
Since $ \lambda^S_\mu(u)=\lambda$, $\sigma_S( u)=1$, we have $\Lambda^S_{\mu}(u)=\lambda$, and consequently,  
 $S_{\lambda,\mu}(u)=S$. From \eqref{LCN}, by direct calculations of $D\lambda^S_\mu(u )$ it follows   
\begin{equation}\label{L=S}
D\lambda^S_\mu(u )=-\frac{1}{Q(u)}DS_{\lambda,\mu}(u), ~~\forall u \in
\Sigma\setminus\{0\}~~\text{such that}~~\Lambda^S_\mu(u )=\lambda,\  \sigma_S(u)=1,	
\end{equation}
and thus, $DS_{\lambda,\mu}(u)=0$.
 
Conversely, assume that $u$
is a weak solution of \eqref{1S} such that
$S=S_{\lambda,\mu}(u)$. Then $\Lambda^S_\mu(u )=\lambda$ and
$D\Lambda^S_\mu(u )=0$. Moreover, by Corollary \ref{cor:diffSigma},
 $\sigma_S(u)=1$, and hence,
$\lambda^S_\mu(u )=\Lambda^S_\mu(u )=\lambda$. Consequently,  \eqref{L=S} yields  $D\lambda^S_\mu(u )=0$. 
\end{proof}

\begin{remark}\label{rem:T(u)>2}
 Lemma \ref{CritR} implies that if $u$ is a weak solution
of \eqref{1S} with prescribed action $S$, then the strong inequality $T(u)>2S$ holds, and thus,   \eqref{Sch} has no standing wave
        $\psi_\lambda=e^{i\lambda t}u$ with action value  $S:=S_{\lambda,\mu}(u)\geq T(u)/2$. 
\end{remark}

\begin{remark}
In view of  the homogeneity of $\lambda^S_\mu(u )$, we may always
assume that any critical point $u$ of $\lambda^S_\mu(u )$ satisfies
$\sigma_S(u)=1$.  
\end{remark}
\emph{From now on, we assume $2<p<q<\infty$.} 
\smallbreak

We look for solutions of Equation~\eqref{1S} in the case
$\lambda > 0$. Therefore  
we are interested in knowing when the condition $\hat{\lambda}^S_\mu:=\sup_{u \in \Sigma_S\setminus
\{0\}}\lambda^S_\mu(u )>0$, $S \in \R$, is satisfied. 
The feasibility of this condition depends on
parameter $\mu$. The limit point of  $\mu$ where
$\hat{\lambda}^S_\mu>0$ is satisfied will be called extreme value of
$\mu$.
 
To find this value, we  reapply the NG-Rayleigh quotient method to the
functional $\lambda^S_\mu(u)$ (see \cite{CarlIlSan}) and therefore
introduce 
\begin{equation*}
	\mu^S(u):=\frac{\(T(u)-2S\)^{1/2}L(u)^{1/2}
          +\frac{1}{q}B(u)}{\frac{1}{p}A(u)},\quad u \in \Sigma_S,
\end{equation*}
which is characterized by the fact that $\mu^S(u)=\mu \Leftrightarrow \lambda^S_\mu(u)=0$, and $\mu^S(u)<\mu \Leftrightarrow \lambda^S_\mu(u)>0$.
For $S\in \R$, define
\begin{equation}\label{eq:hat-mu}
	\hat{\mu}^S:=\inf_{u \in \Sigma_S}\mu^S(u).
\end{equation}

%

Obviously, $\hat{\mu}^S \geq 0$.  
\begin{proposition}\label{prop:2}
Let $N=2$, $2<p<q<\infty$, $S\in \R$. Then $\hat{\mu}^S$ is an extreme
value of $\mu$, that is for any $\mu >\hat{\mu}^S$  condition $\hat{\lambda}^S_\mu>0$
holds, whereas if  $\mu \leq \hat{\mu}^S$, then
$\hat{\lambda}^S_\mu\le 0$.  
\end{proposition}
\begin{proof} Let $\mu >	\hat{\mu}^S$. Then by
  \eqref{eq:hat-mu},  there exists $u
  \in \Sigma_S$ such that  $\hat{\mu}^S<\mu^S(u)<\mu$, and therefore,
  $\lambda^S_\mu(u )>0$  and
    $\hat{\lambda}^S_\mu>0$. 
The converse follows immediately. 	
\end{proof}
Observe that \eqref{eq:hat-mu} can be written in an equivalent form. Indeed, consider for $S\in \R$,
\begin{equation}\label{eq:J}
	M^S(u):=
        \frac{-S+\frac{1}{2}T(u)+\frac{1}{2}L(u)+\frac{1}{q}B(u)}{\frac{1}{p}A(u)},\quad
        u\in \Sigma_S.
\end{equation}
Observe that $M^S(u_\sigma) \to +\infty$ as $\sigma \to 0$ and
$M^S(u_\sigma) \to +\infty$ as $\sigma \to +\infty$. An analysis
similar to that for $\Lambda^S_{\mu}(u)$  shows that 
\begin{equation}\label{eq:M^S0}
	\min_{\sigma>0} M^S(u_\sigma)=\frac{\(T(u)-2S\)^{1/2}L(u)^{1/2} +\frac{1}{q}B(u)}{\frac{1}{p}A(u)} = \mu^S(u). 
\end{equation}
Hence
\[
\hat{\mu}^S=\inf_{u \in \Sigma_S}M^S(u).
\]
\begin{proposition}\label{prop:posit.bound}
	Assume that $S\leq 0$. Then $\hat{\mu}^S>0$.  
\end{proposition}
\begin{proof}
Since for $S\leq 0$,  $M^S(u)\ge M^0(u)$,  
\begin{equation}\label{eq:estb}
 \hat{\mu}^S\geq \inf_{u \in \Sigma\setminus \{0\}}M^0(u),~~~ \forall
S\le 0.
\end{equation}
In view of \eqref{eq:J},  it is easily seen
that for any $u \in \Sigma\setminus \{0\}$ the 
minimum of  $\min_{t>0}M^0(tu)$ is attained at the unique minimizing
point $t_m(u)$ given by

\begin{equation*}
t_m(u)=c_{p,q}\(\frac{T(u)+L(u)}{2B(u)}\)^{1/(q-2)},\quad
 c_{p,q}= \(\frac{(p-2)q}{q-p}\)^{1/(q-2)}.
\end{equation*}
 Hence by \eqref{eq:estb} we have
\[
\hat{\mu}^S\geq \inf_{u \in \Sigma\setminus \{0\}}M^0(t_m(u)u)=C_{p,q}\inf_{u \in \Sigma\setminus \{0\}}
\frac{\(T(u)+L(u)\)^{\frac{q-p}{q-2}}B(u)^{\frac{p-2}{q-2}}}{ A(u)}, ~~ \forall S\leq 0.
\]
for 
\begin{equation*}
  C_{p,q}=\frac{q-2}{p-2}\times\frac{p}{q}\times
  \frac{1}{2^{\frac{q-p}{q-2}}}\times c_{p,q}^{q-p}=
 \frac{(q-2)p}{\(2(q-p)\)^{\frac{q-p}{q-2}}\((p-2)q\)^{\frac{p-2}{q-2}}}.
\end{equation*}
 
By \eqref{eq:uncertainty}, Sobolev and
H\"older inequalities, we have 
\[
\|u\|_{L^p}\leq C_0 \(\|xu\|_{L^2}\|\nabla u\|_{L^2}\)^{\frac{(q-p)}{p(q-2)}} \|u\|_{L^q}^{\frac{q(p-2)}{p(q-2)}}, ~~\forall u \in \Sigma\setminus \{0\}, 
\]
where $C_0>0$ is a constant, and thus, $\hat{\mu}^S>0$ for any $S\leq 0$. 
\end{proof}


\section{Existence of a fundamental frequency solution with	prescribed action}\label{sec:existence}
Recall that $N=2$ and  $2<p<q<\infty$. 
Consider
\begin{equation}\label{FFP}
\hat{\lambda}^S_\mu:=\sup\left\{\lambda^S_\mu(u ); \ u \in \Sigma \setminus
\{0\},\ T(u)-2S\ge 0 \right\},
\end{equation}
with $\mu \geq	\hat{\mu}^S$.  By construction, we may also write
	\begin{equation}\label{FFPG}
\hat{\lambda}^S_\mu=\hat{\Lambda}^S_\mu:=\sup\left\{\Lambda^S_\mu(u
  );\ u
\in \Sigma \setminus \{0\}, \ T(u)-2S\ge 0\right\}.
\end{equation}
In the sequel, we use \eqref{FFP} when we want to take advantage of
the invariance of $\lambda^S_\mu$ with respect to spatial dilations,
$\lambda^S_\mu(u_\sigma)=\lambda^S_\mu(u)$ for all $\sigma>0$.

\begin{lemma}\label{lem:finite}
	Let $S\in \R$. If $\mu \geq	\hat{\mu}^S$, then   
  $0\leq \hat{\lambda}_\mu^S<\infty$. Moreover, if
        $\mu>\hat{\mu}^S$, then  
  $0<\hat{\lambda}_\mu^S<\infty$.
\end{lemma}		

\begin{proof} By Proposition~\ref{prop:2},
  $\hat{\lambda}_\mu^S\ge 0$.
    H\"older inequality yields 
\begin{equation}\label{HSineq}
\|u\|_{L^p} \leq  \|u\|_{L^2}^{1-\theta}\|u\|_{L^q}^{\theta}, \quad
\forall u\in H^1(\R^2), \quad \theta= \frac{q(p-2)}{p(q-2)}\in (0,1),
\end{equation}
where we have used Sobolev embedding to ensure the finiteness of each
term. 
\smallbreak

Suppose that $\hat\lambda_\mu^S=\infty$: there exists a sequence
$(u_n)_n$ in $\Sigma\setminus\{0\}$, with $T(u_n)\ge 2S$, such that
\begin{equation*}
  \lambda_\mu^S(u_n)\Tend n \infty \infty.  
\end{equation*}
In particular, 
\begin{equation}\label{eq:DV}
  \lambda_\mu^S(u_n)\le \frac{1}{\|u_n\|_{L^2}^2}\(
  \frac{\mu}{p}\|u_n\|_{L^p}^p-\frac{1}{q}\|u_n\|_{L^q}^q\) \Tend n \infty \infty.  
\end{equation}
Since $\lambda_\mu^S(u)= \lambda_\mu^S(u_\si)$ for any dilation
parameter $\sigma>0$, we may  assume $\|u_n\|_{L^2}^2=1$ for all
$n$. 
Then
\eqref{HSineq} and \eqref{eq:DV}  yield
\begin{equation*}
  \lambda_\mu^S(u_n)\le
  \frac{\mu}{p}\|u_n\|_{L^p}^p-\frac{1}{q}\|u_n\|_{L^q}^q  \le
  C_1 \|u_n\|_{L^q}^{\theta p} -\frac{1}{q}\|u_n\|_{L^q}^q\Tend
  n\infty \infty,
\end{equation*}
which is impossible,  since $\theta p<q$, and so $\lambda_\mu^S<\infty$. 

It follows from Proposition \ref{prop:2} that for all
  $\mu>\hat{\mu}^S$,  $\hat{\lambda}_\mu^S>0$. 
\end{proof}	
\begin{lemma}\label{lem11}
	Let $S\in \R$. If
        $0<\hat{\lambda}_\mu^S<\infty$, then there exists a maximizer
  $\hat{u}_\mu^S$ of \eqref{FFP}, that is,
  $\hat{\lambda}_\mu^S=\lambda^S_\mu(\hat{u}_\mu^S)$.
\end{lemma}

\begin{proof}
 Consider a maximizing sequence $(u_n)_n$ of \eqref{FFP},
 $\lambda^S_\mu(u_n) \to \hat{\lambda}_\mu^S$ as $n\to+\infty$.  
We show that $(u_n)_n$ is bounded in $\Sigma$. 
\smallbreak

We first assume that $(u_n)_n$ is bounded in $L^2$,
$\|u_n\|_{L^2}\lesssim 1$ (and then show that this is necessarily the case). 
Suppose that $\|u_n\|_{L^q} \to
 +\infty$. 
 Then, in view of \eqref{HSineq},
\begin{align*}\label{Eq1}
-\lambda^S_\mu (u_n )Q(u_n)&= (T(u_n)-2S)^{1/2}
  L(u_n)^{1/2}- \mu   \frac{1}{p}A(u_n)+\frac{1}{q}B(u_n) \\ 
&\ge -\mu \frac{1}{p}\int |u_n|^p+\frac{1}{q}\int |u_n|^q \\ 
&\ge -\mu \frac{1}{p}\|u_n\|_{L^q}^{\theta p}
 +\frac{1}{q}\|u_n\|_{L^q}^q\to +\infty,
          \text{ as } \|u_n\|_{L^q} \to +\infty,  
\end{align*}
which is absurd, since it would imply that $\lambda^S_\mu (u_n )$
becomes negative for large $n$.
Hence  by \eqref{HSineq}, $\|u_n\|_{L^p}$ and $\|u_n\|_{L^q}$  are
bounded. In view of  \eqref{HSineq},
 \begin{equation*}
   \lambda_\mu^S(u_n)\le C\(\frac{\|u_n\|_{L^q}^q}{\|u_n\|_{L^2}^2}\)
     ^{\frac{p-2}{q-2}} - \frac{1}{q} \frac{\|u_n\|_{L^q}^q}{\|u_n\|_{L^2}^2}.
 \end{equation*}
This shows that $\|u_n\|_{L^q}^q/\|u_n\|_{L^2}^2$ is bounded: if we
had $\|u_n\|_{L^2}\to 0$, then we would have $\|u_n\|_{L^q}\to 0$,
hence $\|u_n\|_{L^p}\to 0$. Therefore,
\begin{equation*}
  \hat{\lambda}_\mu^S=\lim_{n\to +\infty}\frac{-1}{Q(u_n)}\left((T(u_n)-2S)^{1/2}L(u_n)^{1/2}-\mu \frac{1}{p}\|u_n\|^p_{L^p}+\frac{1}{q}\|u_n\|^q_{L^q}\right)\le 0,
\end{equation*}
in contradiction with Lemma~\ref{lem:finite}, so $\|u_n\|_{L^2}\gtrsim 1$.
We readily infer that
$(T(u_n)-2S)^{1/2}L(u_n)^{1/2}$ is bounded. Recall the inequality,
found in e.g. \cite[Lemma~4.1]{AnCaSp15}, in the two-dimensional case,
\begin{equation}\label{eq:IMRN}
  \|f\|_{L^2(\R^2)}\le C
  \|f\|_{L^p(\R^2)}^{\theta}\|xf\|_{L^2(\R^2)}^{1-\theta},\quad \theta
  =  \frac{p'}{2}\in (0,1),  \text{where }\frac{1}{p}+\frac{1}{p'}=1.
\end{equation}
This implies that we cannot have $L(u_n)\to 0$, and so $T(u_n)$ is
bounded: ($u_n)_n$ is bounded in $H^1(\R^2)$. If we had $L(u_n)\to \infty$,
then consider
\begin{equation*}
  v_n(x) :=u_n\(xL(u_n)^{1/4}\).
\end{equation*}
By the homogeneity of $\lambda_\mu^S$ with respect to dilations,
$\lambda_\mu^S(v_n)=\lambda_\mu^S(u_n)$, and $v_n $ is a maximizing
sequence. It is such that $T(v_n)=T(u_n)$ is bounded, $L(v_n)=1$, and
\begin{equation*}
  \|v_n\|_{L^2}^2 = L(u_n)^{-1/2}\|u_n\|_{L^2}^2\Tend n \infty 0,
\end{equation*}
while we have seen that this is impossible for a maximizing
sequence. Therefore, $(u_n)_n$ is bounded in $\Sigma$. 	
By the Banach–Alaoglu and Sobolev embedding theorems,  there exists a subsequence, which we again denote by $(u_n)_n$, such that
\begin{align*}
	&u_n \rightharpoonup \hat{u}_\mu^S ~~\mbox{in} ~ \Sigma\\
	&
   u_n \to \hat{u}_\mu^S  ~\mbox{in} ~~
                  L^\gamma(\mathbb{R}^2),~~2\le 
                  \gamma<\infty,\\
	&u_n \to \hat{u}_\mu^S ~~\mbox{a.e. on} ~\mathbb{R}^2,
\end{align*}
for some $\hat{u}_\mu^S \in \Sigma$, as $\Sigma$ is compactly embedded
into $L^\gamma(\R^2)$ for such values of $\gamma$ (see
e.g. \cite[Theorem~XIII.67]{ReedSimon4}).   
We show that $\hat{u}_\mu^S\neq 0$.  To this end, it is sufficient to show that the sequence $\|u_n\|_{L^p}^p$ is separated from zero. Indeed, if  $\|u_n\|_{L^p}^p \to 0$ as $n\to +\infty$, then 
\[
\hat{\lambda}_\mu^S=\lim_{n\to +\infty}\frac{-1}{Q(u_n)}\left((T(u_n)-2S)^{1/2}L(u_n)^{1/2}-\mu \frac{1}{p}\|u_n\|^p_{L^p}+\frac{1}{q}\|u_n\|^q_{L^q}\right)\le 0.
\]
This is impossible, since $\hat{\lambda}_\mu^S>0$.
From this we have
\begin{align*}
		&\lim_{n\to +\infty}\|u_n\|_{L^p}^p=\|\hat{u}_\mu^S\|_{L^p}^p,\\
	&\lim_{n\to +\infty}\|u_n\|_{L^q}^q=\|\hat{u}_\mu^S\|_{L^q}^q,\\
	& \lim_{n\to +\infty}\|u_n\|^2_{L^2}=\|\hat{u}_\mu^S\|^2_{L^2}.
\end{align*}
Recall the Br\'ezis-Lieb lemma (see e.g. \cite{LiebLoss}):
\begin{lemma}\label{BL}
Assume $(v_n)$ is bounded in $L^\gamma(\R^N)$, $1\leq \gamma<+\infty$
	and $v_n \to v$ a. e. on $\R^N$, then
	\begin{equation*}
	\lim_{n\to +\infty}\|v_n\|_{L^\gamma}^\gamma= \|v\|_{L^\gamma}^\gamma+\lim_{n\to +\infty}\|v_n-v\|_{L^\gamma}^\gamma.
	\end{equation*}
\end{lemma}
We infer
\begin{align*}
	&\lim_{n\to +\infty}\|\nabla u_n\|^2_{L^2}=\|\nabla \hat{u}_\mu^S\|^2_{L^2}+\lim_{n\to +\infty}\|\nabla (u_n-\hat{u}_\mu^S)\|^2_{L^2}\geq \|\nabla \hat{u}_\mu^S\|^2_{L^2},\label{Will1}\\
		& \bar{L}:=\lim_{n\to +\infty}\|xu_n\|^2_{L^2}=\|x\hat{u}_\mu^S\|^2_{L^2}+\lim_{n\to +\infty}\|x(u_n-\hat{u}_\mu^S)\|^2\geq \|x\hat{u}_\mu^S\|^2_{L^2}.
\end{align*}
In the case $\|\nabla \hat{u}_\mu^S\|^2_{L^2}\geq 2S$, 
\begin{align*}
	\hat{\lambda}_\mu^S&=\lim_{n\to +\infty}\lambda^S_\mu(u_n)\le \\
	&\frac{-1}{Q(\hat{u}_\mu^S)}\left((T(\hat{u}_\mu^S)-2S)^{1/2}L(\hat{u}_\mu^S)^{1/2}-\mu \frac{1}{p}\|\hat{u}_\mu^S\|_{L^p}^p+\frac{1}{q}\|\hat{u}_\mu^S\|_{L^q}^q \right)=\lambda^S_\mu(\hat{u}_\mu^S),
\end{align*}
which implies that $\hat{u}_\mu^S$ is a maximizer of \eqref{FFP}.  

Suppose,  on the contrary, that $0<\|\nabla
\hat{u}_\mu^S\|^2_{L^2}<2S$ (which might occur only if $S>0$). Then
\begin{align*}
	&\hat{\lambda}_\mu^S=\lim_{n\to +\infty}\lambda^S_\mu(u_n)< \\
&\frac{-\left((T(\hat{u}_\mu^S)+\lim_{n\to +\infty}\|\nabla (u_n-\hat{u}_\mu^S)\|^2_{L^2}-2S)^{1/2}L(\hat{u}_\mu^S)^{1/2}-\mu \frac{1}{p}\|\hat{u}_\mu^S\|_{L^p}^p+\frac{1}{q}\|\hat{u}_\mu^S\|_{L^q}^q \right)}{Q(\hat{u}_\mu^S)}.
\end{align*}
The last term is equal to $\lambda^{S-h}_\mu(\hat{u}_\mu^{S})$,
where $2h=\lim_{n\to +\infty}\|\nabla (u_n-\hat{u}_\mu^S)\|^2_{L^2}$. 
Notice that  if $S_1<S_2$, then $\lambda^{S_1}_\mu(w)\le
\lambda^{S_2}_\mu(w)$ for any $w \in \Sigma_{S_2}$, and consequently  $\hat{\lambda}_\mu^{S_1}\le \hat{\lambda}_\mu^{S_2}$. This implies the contradiction
\[
\hat{\lambda}_\mu^S<\lambda^{S-h}_\mu(\hat{u}_\mu^{S})\le \hat{\lambda}_\mu^S,
\]
and we necessarily have $\|\nabla \hat{u}_\mu^S\|^2_{L^2}\ge 2S$.
\smallbreak

Now we prove that any maximizing sequence is indeed bounded in $L^2$. 
If we had $\|u_n\|_{L^2}\to \infty$, then by the positivity of
$\hat\lambda_\mu^S$,
\begin{equation*}
  \|u_n\|_{L^p}^p\gtrsim \|u_n\|_{L^2}^2,
\end{equation*}
 and together with \eqref{HSineq}, raised to the power $p$, this yields
\begin{equation*}
  \|u_n\|_{L^2}^{2\frac{q-p}{q-2}}\|u_n\|_{L^q}^{q\frac{p-2}{q-2}}\gtrsim
  \|u_n\|_{L^2}^2,\quad\text{hence}\quad \|u_n\|_{L^q}^q\gtrsim
\|u_n\|_{L^2}^2\Tend n\infty \infty.
\end{equation*}
Again due to the positivity of $\hat\lambda_\mu^S$, for sufficiently
large $n$,
\begin{equation*}
  \|u_n\|_{L^p}^p\gtrsim \|u_n\|_{L^q}^q,
\end{equation*}
and \eqref{HSineq} now yields
\begin{equation*}
  \|u_n\|_{L^2}^2\gtrsim \|u_n\|_{L^q}^q,
\end{equation*}
so we come up with
\begin{equation*}
  \|u_n\|_{L^2}^2\approx \|u_n\|_{L^p}^p\approx \|u_n\|_{L^q}^q.
\end{equation*}
Gagliardo-Nirenberg inequality entails $T(u_n)\to \infty$.
Consider the dilation, 
\begin{equation*}
  v_n(x) = u_n\(x\|u_n\|_{L^2}\),
\end{equation*}
and recall that $\lambda_\mu^S(u_n)=\lambda_\mu^S(v_n)$. 
The sequence $(v_n)_n$ is bounded in  $L^2\cap L^q$: from the
previous arguments (where we had considered a maximizing sequence,
bounded in $L^2$), we see that $(v_n)_n$ is bounded in $H^1$. But
$T(u_n)=T(v_n)\to 
\infty$, hence a contradiction. 
\end{proof}
	From Lemmas~\ref{lem:finite} and \ref{lem11}, we get:
\begin{lemma}\label{lem1}
	Let $S\in \R$. If
        $\mu>\hat{\mu}^S$, then there exists a maximizer
  $\hat{u}_\mu^S$ of \eqref{FFP}, that is, $\hat{\lambda}_\mu^S=\lambda^S_\mu(\hat{u}_\mu^S)$.
\end{lemma}	
\begin{remark}
	For the existence of a maximizer
  of \eqref{FFP} in the case $\mu=\hat{\mu}^S$ see Lemma~\ref{lem:CritM} below.
\end{remark}

\begin{lemma}\label{exFFS}
Assume that $\hat u_\mu^S$ is a maximizer of \eqref{FFPG}, i.e., $\hat{\lambda}_\mu^S=\Lambda^S_\mu(\hat{u}_\mu^S)$. Then 	$\hat{u}_\mu^S$ is a
fundamental frequency solution of \eqref{1S} with fundamental frequency $\hat{\lambda}_\mu^S$. 
\end{lemma}
\begin{proof}
 Let $\hat u$ be a maximizer of \eqref{FFPG}. By the Lagrange
	multiple rule, there exist $\mu_0$, $\mu_1$ such that  
	\begin{enumerate}
	\item $|\mu_0| + |\mu_1|\neq 0$,
	\item $\mu_0 D\Lambda^S_\mu(\hat{u} ) + \mu_1 DT(\hat{u})=0,$
	\item  $\mu_1(T(\hat{u})-2S)=0$, $\mu_0\geq 0$, $\mu_1\leq 0$.
	\end{enumerate}
	Evidently, we may assume that $\mu_0=1$. If
	$T(\hat{u})-2S>0$, then by (3), we have $\mu_1 =0$, and consequently,   $D\Lambda^S_\mu(\hat{u} )=0$, and 	thus $\hat{u}$ satisfies \eqref{1S}.
	
	To obtain a contradiction, suppose that $T(\hat{u})-2S=0$.

Using the Pohozaev identity for the equality $D\Lambda^S_\mu(\hat{u} )
+ \mu_1 DT(\hat{u})=0$, we obtain

	\[	\frac{d \Lambda^S_\mu(\hat{u}_\sigma)}{d\sigma}|_{\sigma=1}+\mu_1\frac{d T(\hat{u}_\sigma)}{d\sigma}|_{\sigma=1}=0.
	\]
But $\displaystyle{\frac{d
    \Lambda^S_\mu(\hat{u}_\sigma)}{d\sigma}|_{\sigma=1}<0}$ since $T(\hat{u})-2S=0$, and  $\displaystyle{\frac{d
    T(\hat{u}_\sigma)}{d\sigma}= 0}$.  Thus, $T(\hat{u})-2S>0$ and
$\hat{u}$ satisfies \eqref{1S}.   The equality
$\hat{\lambda}_\mu^S=\Lambda^S_\mu(\hat{u})$ entails that it is  a
fundamental frequency solution.
\end{proof}
	
We can now infer:
\begin{corollary}\label{corINS}
The maximizers $\hat{u}^S_\mu$ of \eqref{FFPG}  satisfy
$T(\hat{u}^S_\mu)-2S>0$.   Up to a gauge transform,
  $\hat{u}_\mu^S$ is radially
    symmetric, nonnegative, and a 
    non-increasing function of $|x|$.
\end{corollary}
\begin{proof}
We explain the last point of the statement. In view of
\cite[Lemma~2.3]{CiJeSe09}, up to replacing $\hat{u}_\mu^S$ by
$e^{i\theta}\hat{u}_\mu^S$ for some $\theta\in \R$, we may assume that
$\hat{u}_\mu^S$ is real-valued. 
Using Schwarz symmetrization (see
e.g. \cite[Chapter~3]{LiebLoss}),  we see that maximizers can be chosen as
radially symmetric, nonnegative, and non-increasing as a function of
$|x|$. It is indeed standard that symmetric rearrangements leave Lebesgue norms unchanged, and do not
increase $T(u)$. They do not increase $L(u)$ either, as shown in
\cite[Appendix~A]{BBJV17}, and more precisely, in the present case
of a harmonic potential $|x|^2$,  \cite[Theorem~4]{BBJV17} asserts
that if $u$ is not equal to its Schwarz symmetrization $u^*$, then
$L(u^*)<L(u)$. Since maximizers satisfy  $T(\hat{u}^S_\mu)>2S$, we see
that if they do not satisfy the announced properties, then considering
their Schwarz symmetrization (strictly) increases the value of
$\lambda^S_\mu$, hence a contradiction.
\end{proof}


\section{Lyapunov stability of the set of fundamental frequency solutions}\label{sec:stability}

We first show why local $\Sigma$ solutions of \eqref{Sch} are
global. In view of e.g. \cite[Theorem~9.2.6]{CazCourant}, it suffices
to prove that any solutions is bounded in $\Sigma$.  For the following
result, we consider the general case $N\ge 1$, with combined
energy-subcritical nonlinearities.

\begin{lemma}\label{lemBOUND}
Let $N\ge 1$. Assume $2<p<q<2^*$ and $\mu>0$. Let $\psi\in
C([0,T);\Sigma)\cap C^1([0,T);\Sigma^*)$ be a solution  of
\eqref{Sch}. Then $\psi$ is  bounded in $\Sigma$: there exists $C$
depending on $\|\psi_{\mid t=0}\|_{\Sigma}$ 
such that for all $t\in [0,T)$, $\|\nabla
\psi(t)\|^2_{L^2}+\|x\psi(t)\|^2_{L^2}\le C$, and therefore we can take
$T=\infty$. 
\end{lemma}
\begin{proof}
  The conservations for \eqref{Sch} yield $Q(\psi(t))=Q(\psi(0))$,
  $H_\mu(\psi(t))=H_\mu(\psi(0))$ for all $t\in [0,T)$. Recall that from \eqref{HSineq},
  there exists $\theta\in ]0,1[$ such that
  \begin{equation*}
    \|\psi(t)\|_{L^p}\le \|\psi(t)\|^{1-\theta}_{L^2}\|\psi(t)\|_{L^q}^\theta
    \lesssim \|\psi(t)\|_{L^q}^\theta,
  \end{equation*}
  since the $L^2$-norm of $\psi$ is independent of time. We infer
  \begin{equation*}
    H_\mu(\psi(0))=H_\mu(\psi(t)) \ge \frac{1}{2}\|\nabla \psi(t)\|_{L^2}^2 +
    \frac{1}{2}\|x  \psi(t)\|_{L^2}^2 -C \|\psi(t)\|_{L^q}^{\theta p} +
    \frac{1}{q} \|\psi(t)\|_{L^q}^q .
  \end{equation*}
  Now from \eqref{HSineq}, $\theta p<q$: this entails that all terms
  in the above inequality are bounded. 
\end{proof}

We now prove (recall that $\F^S_\mu$ is defined in Definition~\ref{def:FFsol}):

\begin{lemma}\label{lemOrbStabFFS}
 Let $N=2$, $2<p<q<\infty$, $S\le 0$, and $\mu>0$ such that $\F^S_\mu\not=\emptyset$.
 Then the set of fundamental frequency solutions $\F^S_\mu$  is
 orbitally stable.
\end{lemma}
\begin{proof}  Suppose that $\F^S_\mu$ is not orbitally stable:
  there exist
  $\varepsilon>0$, a sequence of initial data $\psi_n(0)\in \Sigma$  such that
  \begin{equation}
    \label{eq:CVCI0}
    \inf_{\phi \in \F^S_\mu}\|\psi_n(0)-\phi\|_{\Sigma}\Tend n \infty 0,
  \end{equation}
  and a sequence of times $t_n\ge 0$ along which the solution to
 \eqref{Sch} emanating from $\psi_n(0)$ satisfies
\begin{equation}\label{UNconver}
	 \inf_{\phi \in \F^S_\mu}\|\psi_n(t_n)-\phi\|_{\Sigma} >
         \varepsilon,\quad \forall n\ge 1.
\end{equation}
The convergence \eqref{eq:CVCI0} implies that 
\[
\Lambda_\mu^S(\psi_n(0)) \Tend n \infty \hat{\lambda}_\mu^S,
\]
and therefore  the conservation for $\Lambda^S_\mu$ yields 
\begin{equation}\label{eq:tend}
	\Lambda_\mu^S(\psi_n(t_n)) \Tend n \infty \hat{\lambda}_\mu^S .
\end{equation}
Lemma~\ref{lemBOUND} shows that $(\psi_n(t_n))_n$ is bounded in
$\Sigma$. Consequently, 
there exists a subsequence which we again denote by $(\psi_n(t_n))_n$ such that
\[
\psi_n(t_n)\rightharpoonup u~~\mbox{weakly in }\Sigma,
~~\psi_n(t_n)\to  u,~~\mbox{strongly in }L^p(\R^2), \quad 2\le p<\infty.
\]
for some $u \in \Sigma$. Observe that $T(u)-2S\geq 0$ since $S\le 0$. This and  \eqref{eq:tend} imply that we can 
proceed like in the proof of Lemma~\ref{lem11}, to obtain that we actually
have $\psi_n(t_n)\to u$ strongly in $\Sigma$,   
and 
\[
\lambda_\mu^S(\psi_n(t_n)) \Tend n \infty \hat{\lambda}_\mu^S.
\]
Hence 
 we get  $\Lambda_\mu^S(u)=\hat{\lambda}_\mu^S$, that is $u\in
 \F^S_\mu$. But this contradicts  \eqref{UNconver}.      
\end{proof}


\section{Existence of a minimizer of $\mu^S(u)$}
Assume that $S\leq 0$. Notice that in this case we always have
 $\Sigma_S=\Sigma$. 
 Consider the minimization problem \eqref{eq:hat-mu}, i.e.
\[
	\hat{\mu}^S:=\inf_{u \in \Sigma_S\setminus \{0\}}\mu^S(u).
\]

\begin{lemma}\label{lem:monot}
If $S\le 0$, then there exists a minimizer
  $\bar{v}^S$ of \eqref{eq:hat-mu}. 	
\end{lemma}
\begin{proof} By Proposition~\ref{prop:posit.bound},   $0<\hat{\mu}_S<\infty$.
Let $(v_n)_n$ be a minimizing sequence  of \eqref{eq:hat-mu}, i.e.,
 $\mu^S(v_n) \to \hat{\mu}^S$ as $n\to\infty$ and $v_n \in \Sigma_S$,
 $n\ge 1$.  
We show that $(v_n)_n$ is bounded in $\Sigma$. 
\smallbreak	
Since $\mu^S(v)= \mu^S(v_\si)$ for any dilation
parameter $\sigma>0$,
we may assume $L(v_n)=1$. Since $\mu^S(v_n)$ is bounded, we readily have
\begin{align}
  &0\le \|\nabla v_n\|_{L^2}^2-2S\lesssim
    \|v_n\|_{L^p}^{2p} ,\label{eq:b1}\\
  & \|v_n\|_{L^q}^q\lesssim \|v_n\|_{L^p}^p.\label{eq:b2}
\end{align}
In view of \eqref{HSineq} and \eqref{eq:IMRN},
\begin{equation*}
  \|v_n\|_{L^p}^p\lesssim
  \(\|v_n\|_{L^p}^{p'}\)^{\frac{q-p}{q-2}}\(\|v_n\|_{L^q}^q\)^{\frac{p-2}{q-2}} ,
\end{equation*}
hence
\begin{equation*}
  \(\|v_n\|_{L^p}^p\)^{\frac{q-1}{p-1}}\lesssim
  \|v_n\|_{L^q}^q\lesssim \|v_n\|_{L^p}^p, 
\end{equation*}
where we have used \eqref{eq:b2} for the last inequality. As $2<p<q$,
we infer that $(v_n)_n$ in bounded in $L^p(\R^2)$, and in $L^q(\R^2)$,
again from \eqref{eq:b2}. Now, \eqref{eq:IMRN}  implies that $(v_n)_n$
in bounded in $L^2(\R^2)$, and \eqref{eq:b1} that $(\nabla v_n)_n$
in bounded in $L^2(\R^2)$. Therefore, $(v_n)_n$ is bounded in
$\Sigma$.

Like in the proof of Lemma~\ref{lem11},  there exists a subsequence,
which we again denote by $(v_n)_n$, and  $\bar{v}^S \in \Sigma$, such that
\begin{align*}
	&v_n \rightharpoonup \bar{v}^S ~~\mbox{in} ~ \Sigma\\
	&
   v_n \to \bar{v}^S  ~\mbox{in} ~~
                  L^\gamma(\mathbb{R}^2),~~2\le 
                  \gamma<\infty,\\
	&v_n \to \bar{v}^S ~~\mbox{a.e. on} ~\mathbb{R}^2.
\end{align*}
Observe that $\bar{v}^S \neq 0$. Indeed, by the proof of
Proposition~\ref{prop:posit.bound}, we have

\begin{equation}\label{eq:M^s}
	\mu^S(v_n)\geq \mu^0(t_m(v_n)v_n)=
C_{p,q}
\frac{\(T(v_n)+L(v_n)\)^{\frac{q-p}{q-2}}B(v_n)^{\frac{p-2}{q-2}}}{ A(v_n)},
\end{equation}
where $t_m$ and $C_{p,q}$ were computed in the proof of
Proposition~\ref{prop:posit.bound}.\\

Since $\mu^0(t_m(su_\sigma)su_\sigma)=\mu^0(t_m(u)u)$ for any $\sigma>0$,
$s>0$, $u \in \Sigma\setminus \{0\}$, we may assume that
$T(v_n)L(v_n)=1$, $B(v_n)=1$, for all $n$. Hence if $\bar{v}^S= 0$, we get a contradiction: $\mu^S(v_n)\to +\infty$. 

Now arguing as in the proof of Lemma~\ref{lem11} we obtain
\begin{align*}
	\hat{\mu}^S&=\lim_{n\to +\infty}\mu^S(v_n)\ge \\
	&\frac{1}{A(\bar{v}^S)}\left(\(T(\bar{v}^S)-2S\)^{1/2}L(\bar{v}^S)^{1/2} +\frac{1}{q}B(\bar{v}^S) \right)=\mu^S(\bar{v}^S),
\end{align*}
which implies that $\bar{v}^S$ is a minimizer of \eqref{eq:hat-mu}.  
\end{proof}

\medskip
Observing that  since $\mu^S(u)$ is a homogeneous functional  with respect to spatial dilations, we may always  assume that the minimizer
  $\bar{v}^S$ of \eqref{eq:hat-mu} satisfies $\sigma_S(\bar{v}^S)=1$.
	
\begin{lemma} \label{lem:CritM}
Let $S\leq 0$ and   
  $\bar{v}^S$ be a minimizer \eqref{eq:hat-mu} such that $\sigma_S(\bar{v}^S)=1$. Then  
	\begin{description}
		\item[(i)] $\bar{v}^S$ weakly satisfies \eqref{1S} with 	$\lambda=0$ and $\mu= \hat{\mu}^S$;
		\item[(ii)]  maximizing problem \eqref{FFP} for $\mu= \hat{\mu}^S$ attains its maximizer at $\bar{v}^S$, i.e., $\hat{u}_{\hat{\mu}^S}^S=\bar{v}^S$,   and $\hat{\lambda}^S_{\hat{\mu}^S}=\lambda^S_{\hat{\mu}^S}(\bar{v}^S)=0$.
	\end{description}
\end{lemma}
\begin{proof} 
 First we prove \textbf{(i)}. 
 
If $\bar{v}^S$ is a minimizer of \eqref{eq:hat-mu}, then  
 $D\mu^S(\bar{v}^S)=0$. The assumption $\sigma_S(\bar{v}^S)=1$ implies that $\hat{\mu}^S=\mu^S(\bar{v}^S)=M^S(\bar{v}^S)$. Using this it is easy to check by direct calculations of $D\mu^S(\bar{v}^S)$  that if   $\mu^S(u )=\mu,\  \sigma_S(u)=1$ for $u \in
\Sigma_S$, then 
\begin{equation*}
D\mu^S(u)=\frac{p}{A(u)}DS_{\lambda,\mu}(u)|_{\lambda=0}, 
\end{equation*}
and thus, we get $DS_{0,\hat{\mu}^S}(\bar{v}^S)=0$, that is $\bar{v}^S$ weakly satisfies \eqref{1S} with 	$\lambda=0$ and $\mu= \hat{\mu}^S$.

 We now show \textbf{(ii)}.  If $\bar{v}^S$ is a
minimizer of \eqref{eq:hat-mu}, then the equality
$\hat{\mu}^S=M^S(\bar{v}^S)$ implies that
$\lambda^S_{\hat{\mu}^S}(\bar{v}^S)=0$. By Proposition \ref{prop:2},
$\hat{\lambda}^S_{\hat{\mu}^S}\le 0$, and therefore
$\hat{\lambda}^S_{\hat{\mu}^S}
=\lambda^S_{\hat{\mu}^S}(\bar{v}^S)=0$. 

\end{proof}

%

\begin{proposition}\label{prop:contMUS}
	The function $(-\infty,0]\ni S \mapsto \hat{\mu}^S$ is continuous and  monotone decreasing. Moreover, $
\hat{\mu}^S \to +\infty~~\mbox{as}~~S \to -\infty$ and $
\hat{\mu}^S \to \hat{\mu}^0:=\hat{\mu}^S|_{S=0}>0$ as $S \to 0$. 
\end{proposition}
\begin{proof} 
	Let $S<0$, $\Delta S >0$ such that $S+\Delta S\leq
        0$. Proceeding like in the proof of Lemma~\ref{lem:monot}, we
        see that $(\bar{v}^{S})_{S\in (a,0]}$ is 
        bounded in $\Sigma$ for any bounded $a<0$. Using this it is not hard to show that the following Taylor expansion holds  
\[
	\(T(\bar{v}^S)-2(S+\Delta S)\)^{1/2}=\(T(\bar{v}^S)-2S\)^{1/2}
        -\Delta S \(T(\bar{v}^S)-2S\)^{-1/2}+o(S,\Delta S),
\]
where $o(S,\Delta S)/ |\Delta S| \to 0$ as $\Delta S\to 0$ uniformly
in $S \in (a, 0]$. 	 
	Hence,
	\begin{equation*}
		\hat{\mu}^{S+\Delta S}\leq \mu^{S+\Delta S}(\bar{v}^{S})=\hat{\mu}^{S}-\Delta S \frac{p}{A(\bar{v}^S)}\(T(\bar{v}^S)-2S\)^{-1/2}+o(\Delta S),
	\end{equation*}
	and consequently,
\[
	\hat{\mu}^{S+\Delta S}-\hat{\mu}^{S}\leq -\Delta S \frac{p}{A(\bar{v}^S)}\(T(\bar{v}^S)-2S\)^{-1/2}+o(S,\Delta S).
\]
Similarly we have 
\[
-\Delta S \frac{p}{A(\bar{v}^{S+\Delta S})}\(T(\bar{v}^{S+\Delta S})-2(S+\Delta S)\)^{-1/2}+o(S+\Delta S,\Delta S)\leq	\hat{\mu}^{S+\Delta S}-\hat{\mu}^{S}.
\]
	Thus, $	\hat{\mu}^{S+\Delta S}-\hat{\mu}^{S} \to 0$ as  $\Delta S \to 0$.
	
Observe 
$\hat{\mu}^{S_2}=\mu^{S_2}(\bar{v}^{S_2})<\mu^{S_2}(\bar{v}^{S_1})<\mu^{S_1}(\bar{v}^{S_1})=\hat{\mu}^{S_1}$ for any $S_1 <S_2\leq 0$. Thus, $\hat{\mu}^{S}$ is monotone decreasing. Moreover, it is clear that $
\hat{\mu}^S \to \hat{\mu}^0:=\hat{\mu}^S|_{S=0}$ as $S \to 0$. By Proposition \ref{prop:posit.bound}, $\hat{\mu}^S|_{S=0}>0$ . 

Let us show that $
\hat{\mu}^S \to +\infty~~\mbox{as}~~S \to -\infty$. From the monotonicity of  $\mu^{S}(\bar{v}^{S})$ it follows $\mu^{S}(\bar{v}^{S}) \to C$ as $S \to -\infty$ for some $C\in (0,+\infty]$.  Suppose, contrary to our claim, that $C<+\infty$. As above, we may assume $L(\bar{v}^{S})=1$, $S<0$. Then
\begin{equation}
	\label{eq: Sinfty}
	\mu^S(\bar{v}^{S}):=\frac{\(T(\bar{v}^{S})-2S\)^{1/2} +\frac{1}{q}\|\bar{v}^{S}\|_{L^q}^q}{\frac{1}{p}\|\bar{v}^{S}\|_{L^p}^p} \to C<+\infty~~\mbox{as}~S \to -\infty.
\end{equation}
Then 
\begin{align}
  \|\bar{v}^{S}\|_{L^q}^q\lesssim \|\bar{v}^{S}\|_{L^p}^p.\label{eq:b2F}
\end{align} 
In view of \eqref{HSineq} and \eqref{eq:IMRN},
\begin{equation*}
  \|\bar{v}^{S}\|_{L^p}^p\lesssim
  \(\|\bar{v}^{S}\|_{L^p}^{p'}\)^{\frac{q-p}{q-2}}\(\|\bar{v}^{S}\|_{L^q}^q\)^{\frac{p-2}{q-2}} ,
\end{equation*}
hence by \eqref{eq:b2F},
\begin{equation*}
  \(\|\bar{v}^{S}\|_{L^p}^p\)^{\frac{q-1}{p-1}}\lesssim
  \|\bar{v}^{S}\|_{L^q}^q\lesssim \|\bar{v}^{S}\|_{L^p}^p. 
\end{equation*}
Since $2<p<q$, we infer that $(\bar{v}^{S})$ in bounded in $L^p(\R^2)$. However, $(T(\bar{v}^{S})-2S)\to +\infty$ as 
$S \to -\infty$ and thus, \eqref{eq: Sinfty} implies a contradiction.
\end{proof}
Proposition \ref{prop:contMUS} yields that the range of the function
$(-\infty,0]\ni S \mapsto \hat{\mu}^S$ coincides with
$[\hat{\mu}^{0},+\infty)$ and thus we have:
\begin{corollary}
	There exists an inverse function $ S(\mu)$ of $\hat{\mu}^S$ so that $\hat{\mu}^{S(\mu)}=\mu$ for any $\mu \in [\hat{\mu}^{0},+\infty)$.
\end{corollary}
  
\section{Existence of a ground state}

\begin{proposition}\label{PMon} Assume $\mu>0$, $|S_2-S_1|$
is sufficiently small, and $ S_1<S_2$ such that  $\F^{S_1}_\mu\neq \emptyset$, $\F^{S_2}_\mu\neq \emptyset$. Then 
	\begin{equation}\label{Nerav}
	\frac{S_2-S_1}{Q(\hat{u}^{S_1}_\mu )}\leq\hat{\lambda}_\mu^{S_2} 		- \hat{\lambda}_\mu^{S_1}\leq \frac{S_2-S_1}{Q(\hat{u}^{S_2}_\mu )}, ~~\forall \hat{u}^{S_i}_\mu \in \F^{S_i}_\mu,~i=1,2.
\end{equation}

\end{proposition}
\begin{proof}  Since $\F^{S_j}_\mu \neq \emptyset$,  $\hat{\lambda}_\mu^{S_j}= \Lambda^{S_j}_\mu(\hat{u}^{S_j}_\mu)$,  $\forall \hat{u}^{S_j}_\mu \in \F^{S_j}_\mu,~ j=1,2$. Furthermore, $T( \hat{u}^{S_2}_\mu)>2S_2>2S_1$.  
Hence,  
\begin{align*}
		\hat{\lambda}_\mu^{S_1}= \Lambda^{S_1}_\mu(\hat{u}^{S_1}_\mu)\geq \Lambda^{S_1}_\mu(\hat{u}^{S_2}_\mu)=
		&\Lambda^{S_2}_\mu(\hat{u}^{S_2}_\mu)-\frac{S_2-S_1}{Q(\hat{u}^{S_2}_\mu) }=\hat{\lambda}_\mu^{S_2}-\frac{S_2-S_1}{Q(\hat{u}^{S_2}_\mu) },
\end{align*}
and  we get the second inequality in \eqref{Nerav}. The proof of the
first one 
may be handled in the same way. We only have to notice that
Corollary~\ref{corINS} implies that $T( \hat{u}^{S_1}_\mu)>2S_2$ if
$|S_2-S_1|$ is sufficiently small.   
\end{proof}

\begin{lemma}\label{MuNdepend}
Let $\mu>\hat{\mu}^{0}$. Then $ \F^S_\mu \neq \emptyset$ for any $S\in (S(\mu), 0]$.
\end{lemma}
\begin{proof}
Take $S\in (S(\mu), 0]$. From the above properties,  $\hat{\mu}^{S} >0$. Since $S>S(\mu)$, Proposition~\ref{prop:contMUS} implies $ \mu=\hat{\mu}^{S(\mu)}> \hat{\mu}^{S}$. Hence by Lemma \ref{lem1} we have $ \F^S_\mu \neq \emptyset$.	
\end{proof}
Hence we have
\begin{corollary}\label{corContin}
For any $\mu>\hat{\mu}^{0}$, the   function $S\mapsto \hat{\lambda}_\mu^{S}$ is
continuous and strictly monotone increasing  on $ (S(\mu),0]$,
i.e., $\hat{\lambda}_\mu^{S_2} > \hat{\lambda}_\mu^{S_1}$, for any
$S(\mu)< S_1<S_2\leq 0$. 
Furthermore, $Q(\hat{u}^{S_1}_\mu )> Q(\hat{u}^{S_2}_\mu )$, $\forall
\hat{u}^{S_j}_\mu \in \F^{S_j}_\mu,~ j=1,2$.
\end{corollary}

\begin{lemma}\label{FFS=GS}
Assume that $\mu>\hat{\mu}^{0}$. Suppose  $S\in  (S(\mu),0]$ such that
$\hat{u}^{S}_\mu$ is a fundamental frequency solution  of \eqref{1S}
with $\lambda=\hat{\lambda}_\mu^S$, then
$\hat{u}_{\lambda,\mu}:=\hat{u}^S_\mu$ 	 is a ground state of
\eqref{1S} with ground level $S$.  
\end{lemma}
\begin{proof} 
	Suppose the assertion of the lemma is false. Then there exists a solution $w$ of \eqref{1S} with $\lambda=\hat{\lambda}_\mu^{S}$ such that 
\[
	S_1:=S_{\hat{\lambda}_\mu^{S}, \mu}(w)<S_{\hat{\lambda}_\mu^{S}, \mu}(\hat{u}^{S})=S.
\]
The equality $S_1:=S_{\hat{\lambda}_\mu^{S}, \mu}(w)$ implies
$\Lambda_\mu^{S_1}(w)=\hat{\lambda}_\mu^{S}$.  Moreover, since
$D\Lambda_\mu^{S_1}(w)=0$, Lemma \ref{CritR} (see also Remark
\ref{rem:T(u)>2}) implies that $T(w)-2S_1> 0$. Lemma~\ref{lem:finite} implies that
$\hat{\lambda}_\mu^{S_1}<+\infty$. Furthermore, we have     
\begin{equation}\label{eq:LLineq}
	\hat{\lambda}_\mu^{S_1}=\max_{u \in \Sigma_{S_1}\setminus \{0\} }\Lambda_\mu^{S_1}(u ) \ge \Lambda_\mu^{S_1}(w)= \hat{\lambda}_\mu^{S}>0.
\end{equation}
Thus $0<\hat{\lambda}_\mu^{S_1}<+\infty$, and by Lemma~\ref{lem11}, $\F^{S_1}_\mu\neq \emptyset$. 
Hence, the inequality $S>S_1$ by Proposition~\ref{PMon} implies that $\hat{\lambda}_\mu^{S_1}<\hat{\lambda}_\mu^{S}$, which contradicts  \eqref{eq:LLineq}. 
\end{proof}

\begin{lemma}\label{lem:limit}
Assume that $\mu>\hat{\mu}^{0}$. Then $\lim_{S \to S(\mu)}\hat{\lambda}_\mu^{S} = 0$ and $0<\lim_{S \to 0}\hat{\lambda}_\mu^{S} = \hat{\lambda}_\mu^{*}< \infty $.
\end{lemma}
\begin{proof} Since  $\hat{\mu}^{S(\mu)}=\mu$, Lemma \ref{lem:CritM}  implies that $\lim_{S \to S(\mu)}\hat{\lambda}_\mu^{S}=\hat{\lambda}_\mu^{S(\mu)}=0$.
By the monotonicity of $\hat{\lambda}_\mu^{S}$ it follows that there exists $\lim_{S \to 0}\hat{\lambda}_\mu^{S} = \hat{\lambda}_\mu^{*}\leq \infty  $.		Suppose, contrary to our claim, that  $\hat{\lambda}_\mu^{*}= \infty$. By \eqref{Nerav} and since $Q(\hat{u}^{S_1}_\mu )> Q(\hat{u}^{S_2}_\mu )$, for $S(\mu)< S_1<S_2<0$, this is possible only if  
	$Q(\hat{u}^{S}_\mu ) \to 0~~\mbox{as}~~S \to 0$. 
	An analysis similar to that in the proof of Lemma~\ref{lem:monot} shows that $(\hat{u}^{S}_\mu)_{S\in (\delta,0)}$ is bounded in $\Sigma$ for any $\delta\in (S(\mu),0)$, and thus
	there exists a subsequence $(\hat{u}^{S_n}_\mu)_n$, such that $S_n \to 0$ and
	$\hat{u}^{S_n}_\mu \rightharpoonup \bar{w} ~~\mbox{in} ~ \Sigma$, $\hat{u}^{S_n}_\mu \to \bar{w}$ in	$L^\gamma(\mathbb{R}^2)$, $2\le \gamma<\infty$, $\hat{u}^{S_n}_\mu \to \bar{w}$ a.e. on $\mathbb{R}^2$,	for some $\bar{w} \in \Sigma$. As in proof of Lemma~\ref{lem11}, it follows that $\bar{w}\neq 0$. However, this contradicts  the convergence $Q(\hat{u}^{S_n}_\mu ) \to 0$.
\end{proof}

Corollary~\ref{corContin}  implies that the function $S \mapsto
\hat{\lambda}_\mu^{S}$ is invertible so that for any $\lambda \in
[0, \hat{\lambda}_\mu^{*}]$ there exists a unique
$S_\lambda \in  [S(\mu),0]$ such that
$\hat{\lambda}_\mu^{S_\lambda}=\lambda$.  
 Hence, by Lemmas \ref{FFS=GS}, \ref{lem:limit} we have:
\begin{corollary}\label{EXISTGrS<0}
Assume that $\mu>\hat{\mu}^{0}$. 	Then for any  $\lambda\in [0, \hat{\lambda}_\mu^{*}]$ equation \eqref{1S} has a ground state $\hat{u}_\lambda$.  Moreover, $S_{\lambda,\mu}(\hat{u}_\lambda)<0$ for $\lambda\in [0,
\hat{\lambda}_\mu^{*})$ and $S_{\lambda,\mu}(\hat{u}_\lambda)|_{\lambda=\hat{\lambda}_\mu^{*}}=0$.
\end{corollary}

\begin{lemma}\label{lemGr}
Assume that $\mu>\hat{\mu}^{0}$. 	Suppose  $\lambda\in [0, \hat{\lambda}_\mu^{*}]$ and $\hat{u}_\lambda$ is a ground state   of
\eqref{1S}  with the corresponding ground level $S=S_{\lambda,
  \mu}(\hat{u}_\lambda)\le 0$. Then $\hat{u}_\lambda$ is a fundamental
frequency solution of \eqref{1S}  with fundamental  frequency
$\lambda$. 
\end{lemma}
 \begin{proof} 
Suppose that $\lambda\in [0, \hat{\lambda}_\mu^{*}]$ and $\hat{u}_\lambda$ is a ground state   of
\eqref{1S}  with  $S=S_{\lambda, \mu}(\hat{u}_\lambda)$. Then
$\Lambda^{S}_\mu(\hat{u}_\lambda)=\lambda^{S}_\mu(\hat{u}_\lambda)=\lambda$,
and thus $\lambda\leq \hat{\lambda}^S_\mu$. 

Suppose, contrary to our claim, that  $\lambda< \hat{\lambda}^S_\mu$.
It follows from the above that there exists $S_\lambda \in
(S(\mu),0]$ and a fundamental frequency solution
$\hat{u}^{S_\lambda}_\mu$ such that 
 $\lambda=\hat{\lambda}^{S_\lambda}_\mu=\Lambda^{S_\lambda}_\mu(\hat{u}^{S_\lambda}_\mu)$ 
 and $S_{\lambda, \mu}(\hat{u}^{S_{\lambda}}_\mu)=S_\lambda$.   
Since $\hat{\lambda}^{S_\lambda}_\mu=\lambda< \hat{\lambda}^S_\mu$, Corollary~\ref{corContin}
implies that $S_\lambda<S\le 0$, which contradicts the assumption that
$\hat{u}_\lambda$ is a ground state.  
\end{proof}

\section{Proof of Theorems~\ref{theo:main} and \ref{theo:main2}}

\begin{proof}[Proof of Theorem~\ref{theo:main}]
Let $N=2$, $2<p<q<\infty$ and $\mu>\hat{\mu}^0$. 
Take $\hat{\lambda}_\mu^{*} \in (0,+\infty)$ defined by Lemma~\ref{lem:limit}.

Let $\lambda\in [0, \hat{\lambda}_\mu^{*}]$. Then Corollary~\ref{EXISTGrS<0} yields that
Equation~\eqref{1S} has a ground state $\hat{u}_\lambda \in \Sigma$, and thus we have $(1^o)$.  Moreover, $S_{\lambda,\mu}(\hat{u}_\lambda)<0$ for $\lambda\in [0,
\hat{\lambda}_\mu^{*})$ and
$S_{\lambda,\mu}(\hat{u}_\lambda)|_{\lambda=\hat{\lambda}_\mu^{*}}=0$,
and thus $(2^o)$ holds. 

By Lemma \ref{lemGr}, $\hat{u}_\lambda$ is a fundamental frequency solution of \eqref{1S}  with fundamental  frequency
$\lambda$. This  by Corollaries~\ref{corINS} implies
$(3^o)$. Moreover, we have $\F_{\mu}^S=\G_{\lambda,\mu}$, and thus, by
Lemma \ref{lemOrbStabFFS} we obtain assertion $(4^o)$.
 
\end{proof}
\begin{proof}[Proof of Theorem~\ref{theo:main2}]
Let $N=2$, $2<p<q<\infty$, $S \in \R$.
Proposition \ref{prop:2} implies that the value $\hat{\mu}^S\geq 0 $ given by \eqref{eq:hat-mu} is an extremal value. 
By Lemmas~\ref{lem1} and \ref{exFFS}, it follows that
Equation~\eqref{1S} has a fundamental frequency solution
$\hat{u}^S_\mu \in \Sigma \setminus \{0\}$ for any $\mu >
\hat{\mu}^S$. 

In addition, if $\mu =
\hat{\mu}^S$, the existence of the fundamental frequency solution
$\hat{u}^S_\mu \in \Sigma \setminus \{0\}$ follows from 
Lemma~\ref{lem:CritM}, \textbf{(ii)}. 

Thus, assertion $(1^o)$ is satisfied. Assertion
$(2^o)$ follows from
Corollary~\ref{corINS}. 

If $S\le 0$, then by Proposition \ref{prop:posit.bound}, $\hat{\mu}^S> 0$.   Take  $\mu \ge  \hat{\mu}^S$, then Proposition \ref{prop:contMUS} implies that $S\in (S(\mu), 0]$, and  $\mu \geq  \hat{\mu}^0$. Hence Lemma \ref{FFS=GS} yields that the fundamental frequency solution  $\hat{u}^S_\mu$ is a ground state of \eqref{1S} with $\lambda= \Lambda^S_\mu(\hat{u}^S_\mu)=\hat{\lambda}_\mu^{S}$. Moreover, by the monotonicity of $\hat{\lambda}_\mu^{S}$ and Lemma \ref{lem:limit} it follows $\lambda\in [0,
\hat{\lambda}_\mu^{*})$. Thus we obtain $(3^o)$.

Since $(1^o)$, $\F^S_\mu\not=\emptyset$ for $S\leq 0$, 
 $\mu \geq \hat{\mu}^S$,  
and thus, Lemma~\ref{lemOrbStabFFS} yields that the set of fundamental frequency solutions $\F^S_\mu$  is
 orbitally stable.
\end{proof}

		
\appendix

\section{Proof of Proposition~\ref{prop:constrained}}
\label{sec:constrained}

We resume  the approach introduced in \cite{CGM78} (see also
\cite{BL83a}), and, more precisely, the proof of \cite[Proposition~3.1]{CaDPDE}. 
  In all cases, we consider the constraint $u\in B_p$, where
  \begin{equation*}
  B_p:=\left\{ \phi\in \Sigma,\quad
    \frac{1}{p}\|\phi\|_{L^{p}(\R^N)}^{p}=1\right\}.
\end{equation*}
For the homogeneous nonlinearity, we introduce the functional
\begin{equation*}
  F(u) := \frac{1}{2}\<Hu ,u\>+\frac{\lambda}{2}
   \| u\|_{L^2}^2,\quad  H=-\Delta+|x|^2,
 \end{equation*}
 and
 \begin{equation*}
   \delta = \inf_{u\in B_p}F(u).
 \end{equation*}
We know the spectrum of $H$, $\sigma_p(H) = N+2\N$, and that Hermite
functions provide an $L^2$-eigenbasis; see e.g. \cite{LL}. 
The uncertainty principle also reads
\begin{equation*}
  \<H\phi,\phi\>\ge N\|\phi\|_{L^2}^2,\quad\forall \phi\in \Sigma. 
\end{equation*}
{\bf First case: $\lambda>-N$.} We show that $0<\delta<\infty$. The
finiteness is obvious, and the uncertainty principle implies, for all
$u\in \Sigma$,
\begin{equation}\label{eq:1651}
  F(u) \ge \frac{1}{2}(N+\lambda)\|u\|_{L^2}^2>0.  
\end{equation}
If we had $\delta=0$, then there would exist a minimizing sequence
$u_n\in B_p$ such that $u_n\to 0$ in $L^2$. The definition of
$F$ implies that $u_n\to 0$ in $\Sigma$ too. Our assumption on
$p$ implies $\Sigma\subset H^1(\R^N)\hookrightarrow
L^{p}(\R^N)$, and thus $u_n$ leaves $B_p$ for $n$
sufficiently large, leading to a contradiction: $\delta>0$.

In view of \eqref{eq:1651}, and minimizing sequence is bounded in
$L^2$, and again from the definition of $F$, it is bounded in
$\Sigma$. The Banach–Alaoglu theorem implies that up to a subsequence,
$u_n$ converges weakly to some function $\tilde u\in \Sigma$. The presence of the harmonic potential makes the embedding
$\Sigma \hookrightarrow L^r(\R^N)$ compact for all $2\le
r<2^*$ (see e.g. \cite[Theorem~XIII.67]{ReedSimon4}):
$u_n\to \tilde u$ strongly in $L^2\cap L^{p}$, and $\tilde u\in B_p$
solves
\begin{equation*}
  (H+\lambda)\tilde u = \nu |\tilde u|^{p-2}\tilde u,
\end{equation*}
for some Lagrange multiplier $\nu\in \R$. Since $\delta>0$, taking the
inner product with $\tilde u$ in the above equation shows that
$\nu>0$. The function $u= \nu^{1/(p-2)}\tilde u\in
\Sigma\setminus\{0\}$ then solves 
\eqref{eq:homodefoc}. 
\smallbreak

\noindent {\bf Second case: $\lambda<-N$.} For any $c$, $u_c(x) =c
e^{-|x|^2/2}$ satisfies $Hu_c = Nu_c$, so picking $c>0$ so that
$u_c\in B_p$ shows that $\delta<0$.
\smallbreak

Suppose $\delta=-\infty$:
there would exists a 
sequence $(u_n)_n$ in $B_p$  such that $F(u_n)\to -\infty$. As
$F(u)\ge \frac{\lambda}{2}\|u\|_{L^2}^2$, $(u_n)_n$ is unbounded in
$L^2(\R^N)$.
We remark that $(x u_n)_{n\in \N}$
  and $(\nabla u_n)_{n\in \N}$ are also unbounded in $L^2(\R^N)$,
  with norms of the same order as $\|u_n\|_{L^2}$. Indeed, if we had
  $\|\nabla u_n\|_{L^2}\gg \|u_n\|_{L^2}$ and/or
  $\|xu_n\|_{L^2}\gg \|u_n\|_{L^2}$, then we would have $F(u_n)\ge 0$
  for $n$ sufficiently large. In view of \eqref{eq:uncertainty}, the
  three terms $\|\nabla u_n\|_{L^2}$, $\|xu_n\|_{L^2}$ and
  $\|u_n\|_{L^2}$ go to infinity with the same order of magnitude. 
  Set
\begin{equation*}
    \tilde u_n = \frac{1}{\|u_n\|_{L^2}}u_n.
  \end{equation*}
This is a bounded sequence in $\Sigma$, whose $L^2$ norm is equal to
one. Up to extracting a subsequence, $\tilde u_n$ converges weakly in
$\Sigma$, and strongly in
$L^2\cap L^{p}$ (as we have seen in the first case), to some $\tilde u\in  
B_p$ such that $\|\tilde u\|_{L^2}=1$. We infer
\begin{equation*}
  \|u_n \|_{L^{p}} = \|u_n
  \|_{L^{2}} \|\tilde u_n \|_{L^{p}}\approx \|u_n
  \|_{L^{2}}\to +\infty. 
\end{equation*}
Therefore, $u_n$ cannot remain in $B_p$, hence the finiteness of
$\delta$.
\smallbreak

Knowing that $-\infty<\delta<0$, we infer that any minimizing sequence
is bounded in $\Sigma$, and we conclude like in the first case. The
only difference is that now the Lagrange multiplier $\nu<0$, and we
set $u=|\nu|^{1/(p-2)}\tilde u$. 

\bigbreak

\noindent{\bf Third case: combined nonlinearities, $\lambda>-N$.} The
proof in the inhomogeneous case follows the same lines as the first case. We now consider
  \begin{equation*}
  G(u) := \frac{1}{2}\<Hu ,u\>-\frac{\lambda}{2}
   \|u\|_{L^2}^2+\frac{1}{q}\|u\|_{L^q}^q,
 \end{equation*}
 and the minimization problem
 \begin{equation*}
  \delta=\inf_{u\in B_p}G(u).
\end{equation*}
The uncertainty principle yields
\begin{equation*}
  G(u)\ge \frac{1}{2}(N+\lambda)\|u\|_{L^2}^2 + \frac{1}{q}\|u\|_{L^q}^q,
\end{equation*}
and it is easy to check that the proof of the first case  can be
repeated, up to the final 
homogeneity argument: there exists a Lagrange multiplier $\mu$, which
is positive, but we have essentially no information regarding its
value.

\subsection*{Acknowledgements} The authors wish to thank the referees for 
their careful reading of the paper and their suggestions.

\bibliographystyle{abbrv}
\bibliography{biblio}

\end{document}